\documentclass[11pt,a4paper,reqno]{amsart}
\usepackage{geometry} 
\usepackage{amsmath,amsthm,amssymb}
\usepackage{amsrefs}
\usepackage{hyperref}
\usepackage[latin1]{inputenc}
\usepackage{color}

\newtheorem{definition}{Definition}
\newtheorem{lemma}{Lemma}
\newtheorem{corollary}{Corollary}
\newtheorem{theorem}{Theorem}
\newtheorem{remark}{Remark}

\newcommand{\TT}{\mathbb{T}}
\newcommand{\PP}{\mathbb{P}}
\newcommand{\II}{\mathbb{I}}
\newcommand{\RR}{\mathbb{R}}
\newcommand{\ZZ}{\mathbb{Z}}
\newcommand{\CC}{\mathbb{C}}
\newcommand{\expect}{\mathbb{E}}

\newcommand{\cE}{\mathcal{E}}
\newcommand{\cF}{\mathcal{F}}

\newcommand{\dd}{\mathrm{d}}
\newcommand{\Lz}{{L^2_0(\TT)}}
\newcommand{\cS}{\mathcal{S}}

\newcommand{\cA}{\mathcal{A}}

\newcommand{\cC}{\mathcal{C}}

\newcommand{\eps}{\varepsilon}
\newcommand{\cyl}{\mathcal{C}yl}

\newcommand{\Q}{\ZZ^2\backslash\{0\}}
\newcommand{\QN}{\ZZ^2_N\backslash\{0\}}

\begin{document}

\title[Regularization for SBE]{Regularization by noise and \\  stochastic Burgers equations}
\author{M.~Gubinelli}
\address[M.~Gubinelli]{CEREMADE  UMR 7534 -- Universit\'e Paris--Dauphine}
\email[M.~Gubinelli]{massimiliano.gubinelli@ceremade.dauphine.fr}
\author{M.~Jara}
\address[M.~Jara]{IMPA\\
Estrada Dona Castorina 110\\
CEP 22460-320\\
Rio de Janeiro\\
Brazil}
\email[M.~Jara]{mjara@impa.br}

\begin{abstract}
We study a generalized 1d periodic SPDE of Burgers type:
$$
\partial_t u =- A^\theta u  + \partial_x u^2  + A^{\theta/2} \xi 
$$
where $\theta > 1/2$, $-A$ is the 1d Laplacian,  $\xi$ is a space-time white noise and  the initial condition $u_0$ is taken to be (space) white noise. We introduce a notion of weak solution for this equation in the stationary setting. For these solutions we point out how the noise provide a regularizing effect allowing to prove existence and suitable estimates when $\theta>1/2$. When $\theta>5/4$ we obtain pathwise uniqueness. We discuss the use of the same method to study different approximations of the same equation and for a model of stationary 2d stochastic Navier-Stokes evolution. 
\end{abstract}

\maketitle

The stochastic Burgers equation (SBE) on the one dimensional torus $\TT=(-\pi,\pi]$ is the SPDE
\begin{equation}
\label{eq:burgers}
\dd u_t = \frac12 \partial_\xi^2 u_t(\xi) \dd t + \frac12 \partial_\xi (u_t(\xi))^2 \dd t + \partial_\xi \dd W_t 
\end{equation}
where $W_t$ is a cylindrical white noise on the Hilbert space $H=\Lz$ of square integrable, mean zero real function on $\TT$ and it has the form
$
W_t(\xi) = \sum_{k\in\ZZ_0} e_k(\xi) \beta^k_t
$
with $\ZZ_0 = \ZZ\backslash \{0\}$ and $e_k(\xi)=e^{i k\xi}/\sqrt{2\pi}$ and $\{\beta_t^k\}_{t\ge 0, k\in\ZZ_0}$ is a family of complex Brownian motions such that $(\beta^k_t)^*=\beta^{-k}_t$ and with covariance $\expect[\beta_t^k \beta_t^q ]=\II_{q+k=0}$.   Formally the solution $u$ of eq.~\eqref{eq:burgers} is the derivative of the solution  of the Kardar--Parisi--Zhang equation
\begin{equation}
\label{eq:kpz}
\dd h_t = \frac12 \partial_\xi^2 h_t(\xi) \dd t + \frac12  (\partial_\xi h_t(\xi))^2 \dd t +  \dd W_t 
\end{equation}
which is believed to capture the macroscopic behavior of a large class of surface growth phenomena~\cite{KPZ}.

The main difficulty with eq.~\eqref{eq:burgers} is given by the rough nonlinearity which is incompatible with the distributional nature of the typical trajectories of the process. Note in fact that, at least formally, eq.~\eqref{eq:burgers} preserves the white noise on $H$ and that the square in the non-linearity is almost surely $+\infty$ on the white noise. Additive renormalizations in the form of Wick products are not enough to cure this singularity~\cite{DDT}.

In~\cite{BG} Bertini and Giacomin studying  the scaling limits for the fluctuations of an interacting particles system  show that a particular regularization of~\eqref{eq:burgers}  converges in law to a limiting process $u^{\textrm{hc}}_t(\xi)=\partial_\xi \log Z_t(\xi)$ (which is referred to as the Hopf-Cole solution)  where $Z$ is the solution of the stochastic heat equation with multiplicative space--time white noise
\begin{equation}
\label{eq:she}
\dd Z_t = \frac12 \partial_\xi^2 Z_t(\xi) \dd t + Z_t(\xi) \dd W_t(\xi) .
\end{equation}

The Hopf--Cole solution is believed to be the correct physical solution for~\eqref{eq:burgers} however up to recently a rigorous notion of solution to eq.~\eqref{eq:burgers} was lacking so the issue of uniqueness remained open. 

Jara and Gon\c{c}alves~\cite{JG} introduced a notion of \emph{energy solution} for eq.~\eqref{eq:burgers}  and showed that  the macroscopic current fluctuations of a large class of weakly non-reversible particle systems on $\ZZ$ obey the Burgers equation in this sense. Moreover their results show that also the Hopf-Cole solution is an energy solution of eq.~\eqref{eq:burgers}.

More recently Hairer~\cite{Hairer} obtained a complete existence and uniqueness result for KPZ. In this remarkable paper the theory of controlled rough paths is used to give meaning to the nonlinearity and a careful analysis of the series expansion of the candidate solutions allow to give a consistent meaning to the equation and to obtain a uniqueness result. In particular Hairer's solution coincide with  the Cole-Hopf ansatz.

In this paper we take a different approach to the problem. We want to point out the regularizing effect of the linear stochastic part of the equation on the the non-linear part. This is linked to some similar remarks of Assing~\cite{assing1,assing2} and by the approach of Jara and Gon\c{c}alves~\cite{JG}. Our point of view is motivated also by similar analysis in the PDE and SPDE context where the noise or a dispersive term provide enough regularization to treat some non-linear term: there are examples involving the stochastic transport equation~\cite{FGP}, the periodic Korteweg-de~Vries equation~\cites{kdv,babin-kdv} and the fast rotating Navier-Stokes equation~\cite{babin-ns}. In particular in the paper~\cite{kdv} it is shown how, in the context of the periodic Korteweg-de~Vries equation, an appropriate notion of controlled solution can make sense of the non-linear term in a space of distributions. This point of view has also links with the approach via controlled paths to the theory of rough paths~\cite{controlling}.

With our approach we are not able to obtain uniqueness for the SBE above and we resort to study the more general equation (SBE$_\theta$):
\begin{equation}
\label{eq:burgers-theta}
\dd u_t = - A^\theta u_t \dd t + F(u_t) \dd t + A^{\theta/2} \dd W_t 
\end{equation}
where $F(u_t)(\xi)=\partial_\xi (u_t(\xi))^2$, $-A$ is the Lapacian with periodic b.c., where $\theta\ge 0$ and where the initial condition is taken to be white noise. In the case $\theta=1$ we essentially recover the stationary case of the SBE above (modulo a mismatch in the noise term which do not affect its law). 

For any $\theta \ge 0$ we introduce a class $\mathcal{R}_\theta$ of distributional processes "controlled" by the noise, in the sense that these processes have a  \emph{small time} behaviour similar to that of the stationary  Ornstein-Uhlenbech process $X$  which solves the linear part of the dynamics:
\begin{equation}
\label{eq:ou-theta}
\dd X_t = - A^\theta X_t \dd t + A^{\theta/2} \dd W_t,
\end{equation}
where $X_0$ is white noise.
 When $\theta > 1/2$ we are able to show that the \emph{time integral} of the non-linear term appearing in SBE$_\theta$ is well defined, namely that for all $v\in \mathcal{R}_\theta$ 
\begin{equation}
\label{eq:drift-process}
A^v_t = \int_0 ^t F(v_s) \dd s
\end{equation}
is a well defined process with continous paths in a space of distributions on $\TT$ of specific regularity. Note that this process is not necessarily of finite variation with respect to the time parameter even when tested with smooth test functions. 

The existence of the drift process~\eqref{eq:drift-process} allows to formulate naturally the SBE$_\theta$ equation in the space $\mathcal{R}_\theta$ of controlled processes and gives a notion of solution quite similar to that of energy solution introduced by Jara and Gon\c{c}alves~\cite{JG}. Existence of (probabilistically) weak solutions will be established for any  $\theta > 1/2$, that is well below the KPZ regime. The precise notion of solution will be described below. We are also able to show easily pathwise uniqueness when $\theta > 5/4$ but the  case $\theta=1$ seems still (way) out of range for this technique. In particular the question of pathwise uniqueness is tightly linked with that of existence of strong solutions and the key estimates which will allow us to handle the drift~\eqref{eq:drift-process} are not strong enough to give a control on the difference of two solutions (with the same noise) or on the sequence of Galerkin approximations.

Similar regularization phenomena for stochastic transport equations are studied in~\cite{FGP} and in~\cite{DF} for infinite dimensional SDEs. This is also linked to the fundamental paper of Kipnis and Varadhan~\cite{KV} on CLT for additive functionals and to the Lyons-Zheng representation for diffusions with singular drifts~\cites{MR1988703, MR2065168}.

\bigskip
\textbf{Plan.} In
Sec.~\ref{sec:controlled} we define the class of controlled paths and we recall some results of the stochastic calculus via regularization which are needed to handle the It\^o formula for the controlled processes.
Sec.~\ref{sec:ito-trick} is devoted to introduce our main tool which is a moment estimate of an additive functional of a stationary Dirichlet process in terms of the quadratic variation of suitable forward and backward martingales. In Sec.~\ref{sec:estimates} we use this estimate to provide uniform bounds for the drift of any stationary solution. These bounds are  used in Sec.~\ref{sec:existence} to prove tightness of the approximations when $\theta > 1/2$ and to show existence of controlled solution of the stochastic Burgers equation via Galerkin approximations. Finally in Sec.~\ref{sec:uniq} we prove our pathwise uniqueness result in the case $\theta > 5/4$. In Sec.~\ref{sec:alternative} we discuss related results for the model introduced in~\cite{DDT}.

\bigskip
\textbf{Notations.} We write $X \lesssim_{a,b,\dots} Y$ if there exists a positive constant $C$ depending only on $a,b,\dots$ such that $X \le C Y$. We write $X \sim_{a,b,\dots} Y$ iff $X\lesssim_{a,b,\dots} Y \lesssim_{a,b,\dots} X$. 

We let $\cS$ be the space of smooth test functions on $\TT$, $\cS'$ the space of distributions and $\langle \cdot,\cdot\rangle$ the corresponding duality.

On the Hilbert space $H=\Lz$ the family $\{e_k\}_{k\in\ZZ_0}$ is a complete orthonormal basis. On $H$ we consider the space of smooth cylinder functions  $\cyl$ which depends only on finitely many coordinates on the basis $\{e_k\}_{k\in \ZZ_0}$ and for $\varphi \in\cyl$ we consider the gradient $D \varphi : H\to H$ defined as $D \varphi(x) = \sum_{k\in\ZZ_0} D_k \varphi(x) e_k$ where $D_k = \partial_{x_k}$ and  $x_k = \langle e_k,x \rangle$ are the coordinates of $x$.

For any $\alpha\in \RR$ define the space $\cF L^{p,\alpha}$ of functions on the torus for which
$$
|x|_{\cF L^{p,\alpha}} = \big[\sum_{k\in\ZZ_0} (|k|^\alpha |x_k|)^p\big]^{1/p}<+\infty
\, \text{ 
if $p<\infty$ and 
 }\, 
|x|_{\cF L^{\infty,\alpha}} = \sup_{k\in\ZZ_0} |k|^\alpha |x_k| <+\infty .
$$
We will use the notation $H^\alpha = \cF L^{2,\alpha}$ for the usual Sobolev spaces of periodic functions on $\TT$.
We let $A=-\partial_\xi^2$ and $B=\partial_\xi$ as  unbounded operators acting on  $H$ with domains respectively $H^2$ and $H^{1}$. Note that $\{e_k\}_{k\in\ZZ_0}$ is a basis of eigenvectors of $A$ for which we denote $\{\lambda_k = |k|^2 \}_{k\in\ZZ_0}$ the associated eigenvalues.  The operator $A^\theta$ will then be defined on $H^{\theta}$ by $A^\theta e_k = |k|^{2\theta}e_k$ with domain $H^{2\theta}$.
The linear operator $\Pi_N: H \to H$ is the projection on the subspace generated by $\{e_k\}_{k\in\ZZ_0, |k|\le N}$. 

Denote $\cC_T V = C([0,T],V)$ the space of continuous functions from $[0,T]$ to the Banach space $V$ endowed with the supremum norm and with $\cC^\gamma_T V = C^\gamma([0,T],V)$ the subspace of $\gamma$-H\"older continuous functions in $\cC_T V$ with the $\gamma$-H\"older norm.

\section{Controlled processes}
\label{sec:controlled}

We introduce a space of stationary processes which ``looks like" an Ornstein-Uhlenbeck process. The invariant law at fixed time of these processes will be given by the canonical Gaussian cylindrical measure $\mu$ on $H$ which we consider as a Gaussian measure on $H^{\alpha}$ for any $\alpha<-1/2$. This measure is fully characterized by the equation
$$
\int e^{i \langle \psi,x \rangle}\mu(\dd x) = e^{-\langle \psi,\psi\rangle/2}, \qquad \forall\psi\in H ;
$$
or alternatively by the integration by parts formula
$$
\int D_k \varphi(x) \mu(\dd x) = \int x_{-k} \varphi(x) \mu(\dd x),\qquad  \forall k\in\ZZ_0, \varphi \in\cyl .
$$

\begin{definition}[Controlled process]
\label{def:controlled}
For any $\theta\ge 0$ let $\mathcal{R}_\theta$ be the space of stationary stochastic processes $(u_t)_{0 \leq t \leq T}$ with  continuous paths in $\cS'$ such that 
\begin{itemize}
\item[i)] the law of $u_t$ is the white noise $\mu$ for all $t\in[0,T]$;
\item[ii)] there exists a process $\cA \in C([0,T],\cS')$ of zero quadratic variation such that $\cA_0 = 0$ and satisfying the equation
\begin{equation}
\label{eq:controlled-decomposition}
u_t(\varphi) = u_0(\varphi) + \int_0^t u_s(-A^\theta \varphi) \dd s+\cA_t(\varphi) + M_t(\varphi)
\end{equation}
for any test function $\varphi \in \mathcal S$, where $M_t(\varphi)$ is a martingale with respect to the filtration generated by $u$ with quadratic variation $[M(\varphi)]_t =  2t\|A^{\theta/2} \varphi\|_{L^2_0(\TT)}^2$;
\item[iii)] the reversed processes $\hat u_t = u_{T-t}$, $\hat \cA_t = -\cA_{T-t}$ satisfies the same equation with respect to its own filtration (the backward filtration of $u$).
\end{itemize}
\end{definition}

For controlled processes  we will prove that if $\theta>1/2$ the Burgers drift is well defined by approximating it and passing to the limit. Let  $\rho:\RR\to\RR$ be a  positive smooth test function with unit integral and $\rho^\eps(\xi)=\rho(\xi/\eps)/\eps$ for all $\eps>0$. For simplicity in the proofs we require that the function $\rho$ has a Fourier transform $\hat\rho$ supported in some ball and such that $\hat\rho = 1$ in a smaller ball. This is a technical condition which is easy to remove but we refrain to do so here not to obscure the main line of the arguments. 

\begin{lemma}
\label{lemma:burgers-drift}
If $u\in\mathcal{R}_\theta$ and if $\theta >1/2$ then almost surely
$$
\lim_{\eps\to 0} \int_0^t  F(\rho^\eps* u_s) \dd s
$$
exists in the space $C([0,T],\cF L^{\zeta,\infty})$ for some $\zeta<0$. We \emph{denote} with $\int_0^t  F( u_s) \dd s$ the resulting process  with values in $C([0,T],\cF L^{\zeta,\infty})$.
\end{lemma}
\begin{proof} We postpone the proof in Sect.~\ref{sec:estimates}.
\end{proof}

 It will turn out that for this process we have a good control of its space and time regularity and also some exponential moment estimates.  Then it is relatively natural to \emph{define} solutions of eq.~\eqref{eq:burgers-theta} by the following self-consistency condition.

\begin{definition}[Controlled solution] Let $\theta>1/2$, then a process $u\in\mathcal{R}_\theta$ is a \emph{controlled solution} of   SBE$_\theta$ if almost surely
\begin{equation}
\label{eq:self-consistent}
 \cA_t(\varphi) = \langle \varphi, \int_0^t  F(u_s)  \dd s \rangle
\end{equation}
for any test function $\varphi \in \mathcal S$ and any $t\in[0,T]$.
\end{definition}

Note that these controlled solutions are a generalization of the notion of probabilistically weak solutions of  SBE$_\theta$. The key point is that the drift term is not given explicitly as a function of the solution itself but characterized by the 
self-consistency relation~\eqref{eq:self-consistent}. In this sense controlled solutions are to be understood as a couple $(u,\mathcal{A})$ of processes satisfying compatibility relations.

An analogy which could be familiar to the reader is that with a diffusion on a bounded domain with reflected boundary where the solution is described by a couple of processes $(X,L)$ representing the position of the diffusing particle and its local time at the boundary~\cite{RY}. 

Note also that there is no requirement on $\mathcal{A}$ to be adapted to $u$. Our analysis below cannot exclude the possibility that $\mathcal{A}$ contains some further randomness and that the solutions are strictly weak, that is not adapted to the filtration generated by the martingale term and the initial condition.

\section{The It\^o trick}
\label{sec:ito-trick}

In order to prove the regularization properties of controlled processes we will need some stochastic calculus and in particular an It\^o formula and some estimates for martingales. Let us recall here some basic elements here. In this section $u$ will be always a controlled process in $\mathcal{R}_\theta$. 
For any test function $\varphi\in\cS$ the processes $(u_t(\varphi))_{t}$ and $(\hat u_t(\varphi))_{t}$ are Dirichlet processes: sums of a martingale and a zero quadratic variation process. 
Note that we do not want to assume controlled processes to be semimartingales (even when tested with smooth functions). This is compatible with the regularity of our solutions and there is no clue that solutions of  SBE$_\theta$ even with $\theta=1$ are distributional semimartingales. A suitable notion of stochastic calculus which is valid for a large class of processes and in particular for Dirichlet processes is the stochastic calculus via regularization developed by Russo and Vallois~\cite{RV}. In this approach the It\^o formula can be extended to Dirichlet processes. In particular if $(X^i)_{i=1,\dots,k}$ is an $\RR^k$ valued Dirichlet process and $g$ is a $C^2(\RR^k;\RR)$ function then
$$
g(X_t) = g(X_0) + \sum_{i=1}^k\int_0^t \partial_i g(X_s) \dd^- X^i_s + \frac12 \sum_{i,j=1}^k \int_0^t \partial^2_{i,j} g(X_s) \dd^- [X^i,X^j]_s
$$
where $\dd^-$ denotes the forward integral and $[X,X]$ the quadratic covariation of the vector process $X$. Decomposing $X=M+N$ as the sum of a martingale $M$ and a zero quadratic variation process $N$ we have $[X,X]=[M,M]$ and
$$
g(X_t) = g(X_0) +  \sum_{i=1}^k \int_0^t \partial_i g(X_s) \dd^- M^i_s +  \sum_{i=1}^k \int_0^t \partial_i g(X_s) \dd^- N^i_s
$$
$$ +   \sum_{i,j=1}^k\frac12 \int_0^t \partial^2_{i,j} g(X_s) \dd^- [M^i,M^j]_s
$$
where now $\dd^- M$ coincide with  the usual It\^o integral and $[M,M]$ is the usual quadratic variation of the martingale $M$. The integral $\int_0^t \partial_i g(X_s) \dd^- N^i_s$ is well-defined due to the fact that all the other terms in this formula are well defined. The case the function $g$ depends explicitly on time can be handled by the above formula by considering time as an additional (0-th) component of the process $X$ and using the fact that $[X^i,X^0]=0$ for all $i=1,..,k$. In the computations which follows we will only need to apply the It\^o formula to smooth functions.

Let us denote by $L_0$ the generator of the Ornstein-Uhlenbeck process associated to the operator $A^\theta$:
\begin{equation}
\label{eq:generator-ou}
L_0 \varphi(x) = \sum_{k \in \mathbb Z_0} |k|^{2\theta} \big(- x_k D_k \varphi(x) + \tfrac{1}{2} D_{-k}D_k \varphi(x)\big). 
\end{equation}

Consider now a smooth cylinder function  $h:[0,T]\times \Pi_N H\to \RR$. The It\^o formula for the finite quadratic variation process $(u^N_t = \Pi_N u_t)_t$ gives
$$
h(t,u^N_t)=h(0,u^N_0)+\int_0^t (\partial_s + L^N_0) h(s,u^N_s) \dd s +\int_0^t D h(s,u^N_s) \dd\Pi_N \cA_s + M^+_t 
$$ 
where 
$$
L^N_0 h(s,x) = \sum_{k\in\ZZ_0 : |k|\le N} |k|^{2\theta} ( x_{k} D_k h(s,x) + D_k D_{-k} h(s,x))
$$
is the restriction of the operator $L_0$ to $\Pi_N H$ and where
the martingale part denoted $M^+$ has quadratic variation given by
$
[ M^+ ]_t = \int_0^t \cE^\theta_N(h(s,\cdot))(u^N_s) \dd s
$,
where 
$$
\cE_N^\theta(\varphi)(x) = \frac12  \sum_{k\in \ZZ_0: |k|\le N}|k|^{2\theta}  |D_k \varphi(x)|^2 ,
$$
Similarly the It\^o formula on the backward process reads
$$
h(T-t,u^N_{T-t})=h(T,u^N_T)+ \int_0^{t} (-\partial_s + L^N_0) h(T-s,u^N_{T-s}) \dd s
$$
$$
- \int_0^{t} D h(T-s,u^N_{T-s}) \dd \Pi_N \cA_{T-s} + M^-_t 
$$
with $
[ M^- ]_t = \int_0^t \cE^\theta_N(h(T-s,\cdot))(u^N_{T-s}) \dd s
$
so we have the key equality
\begin{equation}
\label{eq:key-representation}
\int_0^t 2 L_0^N   h(s,u^N_{s})\dd s= -M^+_t + M^-_{T-t}-M^-_T. 
\end{equation}
which allows us to represent the time integral of $h$ as a sum of martingales which allows better control. 
On this martingale representation result we can use the Burkholder--Davis--Gundy inequalities to prove the following bound.
\begin{lemma}[It\^o trick] 
\label{lemma:ito-trick}
Let $h : [0,T]\times \Pi_N H \to \RR$ be a cylinder function. Then
for any $p \geq 1$,
\begin{equation}
\label{eq:ito-trick}
\left\|\sup_{t\in[0,T]}\left|\int_0^t L_0 h(s,\Pi_N u_s) \dd s\right|\right\|_{L^p(\PP_\mu)} \lesssim_p 
T^{1/2} \sup_{s\in[0,T]}\left\|  \cE^\theta(h(s,\cdot))  \right\|^{1/2}_{L^{p/2}(\mu)} 
\end{equation}
where
$
\cE^\theta(\varphi)(x) = \frac12  \sum_{k\in \ZZ_0}|k|^{2\theta}  |D_k \varphi(x)|^2 
$.
In the particular case $h(s,x)= e^{a(T- s)}\tilde h(x)$ for some $a\in\RR$  we have the improved estimate
\begin{equation}
\label{eq:ito-trick-conv}
\begin{split}
\left\|\int_0^T e^{a(T-s)} L_0 \tilde h(\Pi_N u_s) \dd s\right\|_{L^p(\PP_\mu)} 
\lesssim_p
 \left(\frac{1-e^{2aT}}{2a}\right)^{1/2} 
 \left\| \cE^\theta(\tilde h) \right\|^{1/2}_{L^{p/2}(\mu)} .
 \end{split}
\end{equation}
\end{lemma}
\begin{proof}
$$
\left\|\sup_{t\in[0,T]}\left|\int_0^t  2 L_0^N h(s,u_s) \dd s\right|\right\|_{L^p(\PP_\mu)} \le
\left\|\sup_{t\in[0,T]}|M^+_t| \right\|_{L^p(\PP_\mu)}+
2 \left\|\sup_{t\in[0,T]}|M^-_t| \right\|_{L^p(\PP_\mu)}
$$
$$
\lesssim_p \left\| \langle M^+\rangle_T  \right\|_{L^{p/2}(\PP_\mu)}^{1/2}+\left\| \langle M^-\rangle_T  \right\|_{L^{p/2}(\PP_\mu)}^{1/2}
\lesssim_p \left\|\int_0^T \cE^\theta(h(s,\cdot))(u_s) \dd s \right\|_{L^{p/2}(\PP_\mu)}^{1/2}
$$
$$
\lesssim_p \left(\int_0^T\left\| \cE^\theta(h(s,\cdot))(u_s) \right\|_{L^{p/2}(\PP_\mu)} \dd s\right)^{1/2}
\lesssim_p T^{1/2} \sup_{s\in[0,T]} \left\| \cE^\theta(h(s,\cdot)) \right\|_{L^{p/2}(\mu)}^{1/2} .
$$
For the convolution we bound as follows
\begin{equation*}
\begin{split}
\left\|\int_0^T e^{a(T-s)} 2 L_0^N \tilde h(u_s) \dd s\right\|_{L^p(\PP_\mu)} 
&  \lesssim_p \left(\int_0^T e^{2a(T-s)}\dd s\right)^{1/2} 
 \left\| \cE^\theta(\tilde h)(u_0) \right\|^{1/2}_{L^{p/2}(\PP_\mu)} 
 \\ &  \lesssim_p
 \left(\frac{1-e^{2aT}}{2a}\right)^{1/2} 
 \left\| \cE^\theta(\tilde h) \right\|^{1/2}_{L^{p/2}(\mu)}
 \end{split}
\end{equation*}
\end{proof}
The bound~\eqref{eq:ito-trick} in the present form (with the use of the backward martingale to remove the drift part) has been inspired by~\cite{CLO}*{Lemma 4.4}.

\begin{lemma}[Exponential integrability] 
Let  $h : [0,T]\times \Pi_N H \to \RR$ be a cylinder function. Then
\begin{equation}
\label{eq:exp-ito-trick}
\expect \sup_{t\in[0,T]}e^{2 \int_0^t L_0^N h(s,\Pi_N u_s) \dd s} \lesssim 
\expect e^{8 \int_0^T \cE^\theta(h(s,u_s)) \dd s } 
\end{equation}
\end{lemma}
\begin{proof}
Let as above $M^\pm$ be the (Brownian) martingales in the representation of the integral $\int_0^t L_0^N h(s,\Pi_N u_s) \dd s$. By Cauchy-Schwartz
$$
\expect \sup_{t\in[0,T]}e^{2 \int_0^t L_0^N h(s,\Pi_N u_s) \dd s} \le \left[\expect \sup_{t\in[0,T]}e^{2 M^+_t}\right]^{1/2}  \left[\expect \sup_{t\in[0,T]}e^{2 (M^-_T-M^-_{T-t})}\right]^{1/2}.
$$
By  Novikov's criterion 
$
e^{4 M^+_t - 8 \langle M^+\rangle_t }
$ is a martingale for $t\in[0,T]$ if $\expect e^{8 \langle M^+\rangle_T} < \infty$. In this case
$$
\expect \sup_{t\in[0,T]}e^{2 M^+_t} \le \expect \sup_{t\in[0,T]}(e^{2 M^+_t- 4 \langle M^+\rangle_t}\sup_{t\in[0,T]} e^{ 4 \langle M^+\rangle_t })
$$
$$
\le\left[\expect \sup_{t\in[0,T]} e^{4 M^+_t- 8 \langle M^+\rangle_t}\right]^{1/2} \left[\expect e^{8 \langle M^+\rangle_T }\right]^{1/2}
$$
and by Doob's inequality we get that the previous expression is bounded by
$$
\left[\expect e^{4 M^+_T- 8 \langle M^+\rangle_T}\right]^{1/2} \left[\expect e^{8 \langle M^+\rangle_T }\right]^{1/2}
\le \left[\expect e^{8 \langle M^+\rangle_T }\right]^{1/2}.
$$
Reasoning similarly for $M^-$ we obtain that 
$$
\expect \sup_{t\in[0,T]}e^{2 \int_0^t L_0^N h(s,\Pi_N u_s) \dd s} \le \expect e^{8 \langle M^+\rangle_T } = \expect e^{8 \int_0^T \cE^\theta(h(s,u_s)) \dd s }.
$$
\end{proof}

\section{Estimates on the Burgers drift}
\label{sec:estimates}

In this section we provide the key estimates on the Burgers drift via the quadratic variations of the forward and backward martingales in its decomposition.
Let $F(x)(\xi) = B (x(\xi))^2$ and $F_N(x) = F(\Pi_N x)$. Define 
$$
H_N(x) = -\int_0^\infty F_N(e^{-A^\theta t}x)\dd t
$$
and consider $L_0 H_N(x)$ as acting on each Fourier coordinate of $H_N(x)$. Remark that the second order part of $L_0$ does not appear in the computation of $L_0 F_N$ since 
$$
D_k D_{-k} F(\Pi_N e^{-A^\theta t} x)=0
$$ for each $k\in\ZZ_0$. Indeed
$$
D_{-k} D_k F(\Pi_N e^{-A^\theta t} x) = B [D_{-k} D_k (\Pi_N e^{-A^\theta t} x)^2]=2 B D_{-k} [(\Pi_N e^{-A^\theta t} x) (\Pi_N e^{-A^\theta t} e_k)]
$$
$$
=2   [B(\Pi_N e^{-A^\theta t} e_{-k}) (\Pi_N e^{-A^\theta t} e_k)+(\Pi_N e^{-A^\theta t} e_{-k}) B(\Pi_N e^{-A^\theta t} e_k)] = 0
$$
Then it is easy to check that 
$$
L_0 H_N(\Pi_N x) = \langle A^\theta x, D H_N(\Pi_N x)\rangle  = -2 \int_0^\infty B [(e^{-A^\theta t}\Pi_N x)(A^\theta e^{-A^\theta t} \Pi_N x) ]\dd t
$$
$$
=  -\int_0^\infty \frac{\dd }{\dd t}B [(e^{-A^\theta t}\Pi_N x)^2 ] = B (\Pi_N x)^2=F(\Pi_N x)
$$
since $\lim_{t\to \infty} B [(e^{-A^\theta t}\Pi_N x)^2 ]  = 0$. Denote by $(x_k)_{k\in\ZZ}$ and $(H_N(x)_k)_{k\in\ZZ_0}$ the coordinates of $x=\sum_{k\in\ZZ_0} x_k e_k$ and $H_N(x)=\sum_{k\in\ZZ_0} H_N(x)_k e_k$  in the canonical basis $(e_k)_{k\in\ZZ_0}$. Then a direct computation gives an explicit formula for $H_N(x)$:
$$
(H_{N}(x))_k =
2 ik  \sum_{k_1,k_2 : k=k_1+k_2}  \frac{\II_{|k|,|k_1|,|k_2|\le N}}{|k_1|^{2\theta}+|k_2|^{2\theta}} x_{k_1} x_{k_2}.
$$
Let us denote with $(H_{N}(x))_k^{\pm}$ respectively the real and imaginary parts of this quantity: $(H_{N}(x))_k^{\pm}= ((H_{N}(x))_k\pm (H_{N}(x))_{-k})/(2 i^{\pm})$ where $i^+=1$ and $i^-=i$. Now
$$
(H_{N}(x))^\pm_k =
 i^{\mp}k  \sum_{k_1,k_2 : k=k_1+k_2}  \frac{\II_{|k|,|k_1|,|k_2|\le N}}{k_1^{2\theta}+k_2^{2\theta}} (x_{k_1} x_{k_2}\mp x_{-k_1} x_{-k_2})
$$
and recall that
$
\cE^\theta((H_N)^\pm_k)(x) = \sum_{q\in\ZZ_0} |q|^{2\theta} |D_q H^\pm_{N,k}(x)|^2
$.
\begin{lemma}
\label{lemma:energy-estimates}
For $\lambda >0$ small enough we have
\begin{equation}
\label{eq:first-energy-estimate}
\sup_{k\in\ZZ_0} \expect \exp\left[\lambda |k|^{2\theta-3}  \cE^\theta((H_N)^\pm_k)(u_0)\right] \lesssim 1
\end{equation}
and
\begin{equation}
\label{eq:second-energy-estimate}
\sup_{1\le M \le N}\sup_{k\in\ZZ_0} \expect \exp\left[\lambda |k|^{-2} M^{2\theta-1}  \cE^\theta((H_N-H_M)^\pm_k)(u_0)\right] \lesssim 1.
\end{equation}
\end{lemma}
\begin{proof}
We start by computing $\cE((H_N)^\pm_k)$: noting that
$$
D_q (H_{N})^\pm_k(x)=  i^\mp k  \left[  \frac{\II_{|k|,|q|,|k-q|\le N}}{|q|^{2\theta}+|k-q|^{2\theta}}  x_{k-q}\mp \frac{\II_{|k|,|q|,|k+q|\le N}}{|q|^{2\theta}+|k+q|^{2\theta}}  x_{-k-q}\right]
$$ 
we have
\begin{equation*}
\begin{split}
 \cE^\theta((H_N)^\pm_k)(x) & = \sum_{q\in\ZZ_0} |k|^2 |q|^{2\theta}\left[
 2 \frac{\II_{|k|,|q|,|k-q|\le N}}{(|q|^{2\theta}+|k-q|^{2\theta})^2}  |x_{k-q}|^2
 \right . \\
  & \left .
 \qquad \qquad \mp \frac{\II_{|k|,|q|,|k-q|\le N}}{|q|^{2\theta}+|k-q|^{2\theta}}  \frac{\II_{|k|,|q|,|k+q|\le N}}{|q|^{2\theta}+|k+q|^{2\theta}}  (x_{k-q} x_{k+q}+x_{-k+q} x_{-k-q})
 \right]
\end{split}
\end{equation*}
which gives the bound
$$
 \cE^\theta((H_N)^\pm_k)(x)\lesssim   |k|^2  \sum_{\substack{k_1,k_2 :k_1+k_2=k\\|k|,|k_1|,|k_2|\le N}} \frac{|k_1|^{2\theta}\II_{|k|,|k_1|,|k_2|\le N}}{(|k_1|^{2\theta}+|k_2|^{2\theta})^2}  |x_{k_2}|^2 
$$
$$
\lesssim  |k|^2  \sum_{\substack{k_1,k_2 :k_1+k_2=k\\|k|,|k_1|,|k_2|\le N}} \frac{\II_{|k|,|k_1|,|k_2|\le N}}{|k_1|^{2\theta}+|k_2|^{2\theta}}  |x_{k_2}|^2 =   \sum_{\substack{k_1,k_2 :k_1+k_2=k\\|k|,|k_1|,|k_2|\le N}}c(k,k_1,k_2) |x_{k_2}|^2  = h_N(x)
$$
where $c(k,k_1,k_2) = |k|^2/(|k_1|^{2\theta}+|k_2|^{2\theta})$.
Let $$
I_N(k) =  \sum_{\substack{k_1,k_2 :k_1+k_2=k\\|k|,|k_1|,|k_2|\le N}}  c(k,k_1,k_2) 
$$
and note that the sum in $I_N(k)$ can be bounded by the equivalent integral giving (uniformly in $N$)
$$
I_N(k) \lesssim  |k|^{2}  \int_{\RR} \frac{\dd q}{|q|^{2\theta}+|k-q|^{2\theta}} 
= |k|^{3-2\theta}  \int_{\RR} \frac{\dd q}{|q|^{2\theta}+|1-q|^{2\theta}} \lesssim |k|^{3-2\theta} 
$$
since that the last integral is finite for $\theta > 1/2$. Then
$$
\expect e^{\lambda |k|^{2\theta-3}\cE^\theta((H_N)^\pm_k)(u_0)} \le \expect e^{\lambda C|k|^{2\theta-3} h_N(u_0)}
$$
$$
 \le \sum_{\substack{k_1,k_2 :k_1+k_2=k\\|k|,|k_1|,|k_2|\le N}} c(k,k_1,k_2)  \expect \frac{e^{\lambda C |k|^{2\theta-3} I_N(k) |(u_0)_{k_2}|^2}}{I_N(k)}
\le \sum_{\substack{k_1,k_2 :k_1+k_2=k\\|k|,|k_1|,|k_2|\le N}} c(k,k_1,k_2)  \expect \frac{e^{\lambda C'|(u_0)_{k_2}|^2}}{I_N(k)} 
$$
where we used the previous bound to say that $C|k|^{2\theta-3} I_N(k) \le C'$ uniformly in $k$. Remind that $(u_0)_k$ has a Gaussian distribution of mean zero and unit variance. Therefore for $\lambda$ small enough $\expect e^{\lambda C'|(u_0)_{k_2}|^2}\lesssim 1$ uniformly in $k_2$ so that  
$$
\expect e^{\lambda |k|^{2\theta-3}\cE^\theta((H_N)^\pm_k)(u_0)} \lesssim 1.
$$
This establishes the claimed exponential bound for $\cE^\theta((H_N(x))_k^\pm)$. 
Similarly we have
$$
\cE^\theta((H_N-H_M)^\pm_k)(x)  \lesssim \sum_{k_1,k_2 :k_1+k_2=k} (\II_{|k|,|k_1|,|k_2|\le N}-\II_{|k|,|k_1|,|k_2|\le M})^2 c(k,k_1,k_2) | x_{k_2}|^2 .
 $$
Let
$$
I_{N,M}(k) =\sum_{k_1,k_2 :k_1+k_2=k} (\II_{|k|,|k_1|,|k_2|\le N}-\II_{|k|,|k_1|,|k_2|\le M})^2 c(k,k_1,k_2)
$$
and note that, for $N\ge M$,
$$
(\II_{|k|,|k_1|,|k_2|\le N}-\II_{|k|,|k_1|,|k_2|\le M}) \lesssim \II_{|k|,|k_1|,|k_2|\le N}(\II_{|k|> M}+\II_{|k_1|> M}+\II_{|k_2|> M}).
$$
Then, by estimating the sums with the corresponding integrals and after easy simplifications we remain with the following bound
$$
I_{N,M}(k)\lesssim |k|^{2} \II_{|k|> M} \int_{\RR} \frac{\dd q}{|q|^{2\theta}+|k-q|^{2\theta}} 
+ |k|^{2}  \int_{\RR} \frac{\II_{|q|> M}\dd q}{|q|^{2\theta}+|k-q|^{2\theta}} 
$$
The first integral in the r.h.s.~is easily handled by
$$
|k|^{2} \II_{|k|> M} \int_{\RR} \frac{\dd q}{|q|^{2\theta}+|k-q|^{2\theta}} \lesssim |k|^{3-2\theta} \II_{|k|> M} \lesssim |k|^2 M^{1-2\theta}
$$
since $\theta > 1/2$. For the second we have the analogous bound
$$
|k|^{2}  \int_{\RR} \frac{\II_{|q|> M}\dd q}{|q|^{2\theta}+|k-q|^{2\theta}} \lesssim 
|k|^{2}  \int_{\RR} \frac{\II_{|q|> M}\dd q}{|q|^{2\theta}} \lesssim |k|^2 M^{1-2\theta}
$$
which concludes the proof.
\end{proof}

Using Lemma~\ref{lemma:ito-trick} and the estimates contained in Lemma~\ref{lemma:energy-estimates} we are led to the next set of more refined estimates for the drift and his small scale contributions.

\begin{lemma}
\label{lemma:main-bounds}
Let 
$
G^M_t = \int_0^t F_{M}(u_s) \dd  s
$.
 For any $M\le N$ we have
\begin{equation}
\label{eq:basic-est-1}
\|\sup_{t\in[0,T]} \left|(G^M_t)_k\right| \|_{L^p(\PP_\mu)}\lesssim_p |k| M T ,
\end{equation}
\begin{equation}
\label{eq:basic-est-2}
\|\sup_{t\in[0,T]} \left|(G^M_t)_k\right| \|_{L^p(\PP_\mu)}\lesssim_p |k|^{3/2-\theta} T^{1/2} ,
\end{equation}
\begin{equation}
\label{eq:basic-est-3}
\|\sup_{t\in[0,T]}\left|(G^M_t)_k-(G^N_t)_k\right| \|_{L^p(\PP_\mu)}\lesssim_p |k| T^{1/2} M^{1/2-\theta} ,
\end{equation}
\begin{equation}
\label{eq:basic-est-4}
\sup_{M\ge 0}\|\sup_{t\in[0,T]}\left|(G^M_t)_k\right| \|_{L^p(\PP_\mu)}\lesssim_p   |k| T^{2\theta/(1+2\theta)} .
\end{equation}
\end{lemma}
\begin{proof}
The Gaussian measure $\mu$ satisfies the hypercontractivity estimate (see for example~\cite{janson_gaussian_1997}): for any complex-valued finite order polynomial $P(x)\in\cyl$ we have
\begin{equation}
\label{eq:hypercontractivity}
\left\| P(x) \right\|_{L^p(\mu)}\lesssim_p \left\| P(x) \right\|_{L^2(\mu)}.
\end{equation}
Then we have $(F_M(x))_k = ik \sum_{k_1+k_2=k} x_{k_1} x_{k_2}$ and for all $k\neq 0$
$$
\int |(F_M(x))_k|^2 \mu(\dd x) = |k|^2 \sum_{k_1+k_2=k} \sum_{k'_1+k'_2=k} \II_{|k_1|,|k_2|,|k'_1|,|k'_2|\le M}\int x_{k_1} x_{k_2} x_{k'_1}^* x_{k'_2}^* \mu(\dd x)  
$$
$$
=4 |k|^2 M^2
$$
This allows us to obtain the bound~\eqref{eq:basic-est-1}. Indeed
$$
\|\sup_{t\in[0,T]}\left| (G^M_t)_k \right| \|_{L^p(\PP_\mu)}
\lesssim 
\int_0^T \left\| (F_{M}(u_s))_k \right\|_{L^p(\PP_\mu)}\dd  s 
$$
$$\lesssim 
T \left\| (F_{M}(\cdot))_k  \right\|_{L^p(\mu)}
\lesssim_p T \left\| (F_{M}(\cdot))_k  \right\|_{L^2(\mu)} \lesssim_p
  |k| M T.
$$
For the bound~\eqref{eq:basic-est-2} we use the fact that $L_0 H_N = F_N$ and Lemma~\ref{lemma:ito-trick}  to get 
\begin{equation*}
\|\sup_{t\in[0,T]} \left|(G^M_t)_k\right| \|_{L^p(\PP_\mu)}\lesssim_p T^{1/2} \sup_{t\in[0,T]}
\| \cE^\theta(H_N(\cdot)) \|_{L^{p/2}(\mu)}^{1/2}
\lesssim
|k|^{3/2-\theta} T^{1/2}
\end{equation*}
where we used the first energy estimate~\eqref{eq:first-energy-estimate}
 of Lemma~\ref{lemma:energy-estimates} and the fact that
$
\|Q\|_{L^p(\mu)}^p \lesssim_p \int [e^{Q(x)^+}+e^{Q(x)^-}] \mu(\dd x)
$
where again $Q^\pm$ are the real and imaginary parts of $Q$. 
The bound~\eqref{eq:basic-est-3} is obtained in the same way using the second energy estimate~\eqref{eq:second-energy-estimate}. 
Finally  the last bound~\eqref{eq:basic-est-4} is obtained from the previous two by taking $0\le N\le M$, decomposing $F_M(x) = F_N(x)-F_{N,M}(x)$:
$$
\|\sup_{t\in[0,T]} \left|(G^M_t)_k\right| \|_{L^p(\PP_\mu)} \le \|\sup_{t\in[0,T]} \left|(G^N_t)_k\right| \|_{L^p(\PP_\mu)}+\|\sup_{t\in[0,T]} \left|(G^M_t)_k-(G^N_t)_k\right| \|_{L^p(\PP_\mu)}
$$
$$
\lesssim_p |k|  ( N T+  N^{1/2-\theta} T^{1/2})
$$
and performing the optimal choice $N \sim T^{-1/(1+2\theta)}$.
\end{proof}

Analogous estimates go through also for the functions obtained via convolution with the $e^{-A^\theta t}$ semi-group.
\begin{lemma}
\label{lemma:main-bounds-conv}
Let
$$
\tilde G^M_t = \int_0^t e^{-A^\theta (t-s)} F_{M}(u_s) \dd  s
$$
then for any $M\le N$ we have
\begin{equation}
\label{eq:basic-est-1-conv}
\|(\tilde G^M_t)_k\|_{L^p(\PP_\mu)}\lesssim_p |k| M \left(\frac{1-e^{-2k^{2\theta} t/2}}{2k^{2\theta}}\right)
\end{equation}
\begin{equation}
\label{eq:basic-est-2-conv}
\|(\tilde G^M_t)_k\|_{L^p(\PP_\mu)}\lesssim_p |k|^{3/2-\theta} \left(\frac{1-e^{-2k^{2\theta} t/2}}{2k^{2\theta}}\right)^{1/2}
\end{equation}
\begin{equation}
\label{eq:basic-est-3-conv}
\|(\tilde G^M_t)_k-(\tilde G^N_t)_k\|_{L^p(\PP_\mu)}\lesssim_p |k| M^{1/2-\theta} \left(\frac{1-e^{-2k^{2\theta} t/2}}{2k^{2\theta}}\right)^{1/2}
\end{equation}
\end{lemma}
\begin{proof}
The proof follows the line of Lemma~\ref{lemma:main-bounds} using eq.~\eqref{eq:ito-trick-conv} instead of eq.~\eqref{eq:ito-trick}.  
\end{proof}

\begin{corollary}
For all sufficiently small $\eps > 0$
\begin{equation}
\label{eq:basic-est-4-conv}
\sup_{N\ge 0}\|(\tilde G^N_t)_k - (\tilde G^N_s)_k\|_{L^p(\PP_\mu)} \lesssim_{p} |k|^{3/2-2\theta+2\eps \theta} (t-s)^\eps
\end{equation}
\end{corollary}
\begin{proof}
To control the time regularity of the drift convolution we consider $0\le s \le t$ and decompose
$$
\|(\tilde G^N_t)_k - (\tilde G^N_s)_k\|_{L^p(\PP_\mu)}
$$
$$
\le \| \int_s^t (e^{-A^\theta(t-r)} F_{N}(u_r))_k \dd  r\|_{L^p(\PP_\mu)} + (e^{-k^{2\theta}(t-s)}-1)\|(\tilde G^N_s)_k\|_{L^p(\PP_\mu)}
$$
$$
\lesssim |k|^{3/2-\theta} (t-s)^{1/2} +|k|^{3/2-2\theta}(e^{-k^{2\theta}(t-s)}-1)
\lesssim |k|^{3/2-\theta} (t-s)^{1/2}
$$
Moreover a direct consequence of eq.~\eqref{eq:basic-est-2-conv} is
$$
\sup_{t\in[0,T]} \|(\tilde G^N_t)_k \|_{L^p(\PP_\mu)}\lesssim_p |k|^{3/2-2\theta}.
$$
which give us a uniform estimate in the form
$$
\|(\tilde G^N_t)_k - (\tilde G^N_s)_k\|_{L^p(\PP_\mu)} \le \|(\tilde G^N_t)_k\|_{L^p(\PP_\mu)}+\|(\tilde G^N_s)_k\|_{L^p(\PP_\mu)}   \lesssim_{p} |k|^{3/2-2\theta}
$$
By interpolation we get the claimed bound. 
\end{proof}

\begin{remark}
All these $L^p$ estimates can be replaced with equivalent exponential estimates. For example it is not difficult to prove that for small $\lambda$ we have
$$
\sup_{t\in[0,T]} \sup_{k\in\ZZ_0} \expect \exp\left(\lambda |k|^{2\theta-3/2} (\tilde G^N_t)^\pm_k \right) \lesssim 1
$$
where $(\cdot)^\pm$ denote, as before, the real and imaginary parts, respectively.
\end{remark} 

At this point we are in position to prove Lemma~\ref{lemma:burgers-drift} on the existence of the Burgers' drift for controlled processes.

\begin{proof}(of Lemma~\ref{lemma:burgers-drift})
Let $\mathcal{B}^\eps_t = \int_0^t  F(\rho^\eps* u_s) \dd s$. 
We start by noting that since $\hat \rho$ has a bounded support we have $\rho^\eps * (\Pi_N u_s) = \rho^\eps * u_s$ for  all $N \ge C/\eps$ for some constant $C$ and $\eps$ small enough. Moreover all the computation we made for $F_N$ remains true for the functions $F_{\eps,N}(x) = F(\rho^\eps * \Pi_N x)$ so we have estimates analogous to that in  Lemma~\ref{lemma:main-bounds} for 
$G^{\eps,M}_t = \int_0^t \int_0^t  F(\rho^\eps* \Pi_M u_s) \dd s$. In taking $\eps>\eps'>0$ and $N\ge C/\eps$, $M\ge C/\eps'$ and $M\ge N$ we have
$$
\left\|\sup_{t\in[0,T]}\left|(\mathcal{B}^\eps_t)_k-(\mathcal{B}^{\eps'}_t)_k\right| \right\|_{L^p(\PP_\mu)}
=
\left\|\sup_{t\in[0,T]}\left|(G^{\eps,N}_t)_k-(G^{\eps',M}_t)_k\right|\right \|_{L^p(\PP_\mu)}
$$
$$
\lesssim_p |k| T^{1/2} M^{1/2-\theta} \lesssim_p |k| T^{1/2} (\eps')^{\theta-1/2} 
$$
uniformly in $\eps,\eps',N,M$. This easily implies that the sequence of processes $(\mathcal{B}^\eps)_{\eps}$ converges almost surely to a limit in $C(\RR_+,\cF L^{-1-\eps,\infty})$ if $\theta>1/2$. By similar arguments it can be shown that the limit does not depend on the function $\rho$. 
\end{proof}

\section{Existence of controlled solutions}
\label{sec:existence}

Fix $\alpha < 1/2$ and consider the SDE on $H^\alpha$ given by
\begin{equation}
\label{eq:burgers-reg}
\dd u^N_t = - A^\theta u^N_t  \dd t + F_N(u^N_t)\dd t + A^{\theta/2} \dd W_t,
\end{equation}
where  $F_N : H\to H$ is defined by $F_N(x) = \frac12  \Pi_N B (\Pi_N x)^2$. Global solution of this equation starting from any $u_0^N\in H^\alpha$  can be constructed as follows.  Let $(Z_t)_{t\ge 0}$ the unique OU process on $H^\alpha$ which satisfies the SDE
\begin{equation}
\label{eq:ou}
\dd Z_t = - A^\theta Z_t  \dd t + A^{\theta/2} \dd W_t.
\end{equation}
with initial condition $Z_0 = u^N_0$. Let $(v^N_t)_{t\ge 0}$ the unique solution taking values in the finite dimensional vector space $\Pi_N H$ of the following SDE
$$
\dd v^N_t = - A^\theta v^N_t  \dd t + F_N(v^N_t)\dd t + A^{\theta/2}\dd  \Pi_N  W_t,
$$
with initial condition $v^N_0 = \Pi_N u^N_0$. Note that this SDE has global solutions despite of the quadratic non-linearity. Indeed the vector field $F_N$ preserves the $H$ norm: 
$$
\langle v^N_t,F_N(v^N_t)\rangle = \langle v^N_t,B (v^N_t)^2\rangle = \frac13 \int_\TT \partial_\xi(v^N_t(\xi))\dd \xi = 0 
$$
and by It\^o formula we have
$$
\dd \|v^N_t\|_H^2 = 2 \langle v^N_t,- A^\theta v^N_t  \dd t + F_N(v^N_t)\dd t + A^{\theta/2}\dd  \Pi_N  W_t \rangle + C_N \dd t
$$
$$
  = -2  \|A^{\theta/2} v^N_t\|^2_H \dd t + 2 \langle v^N_t, A^{\theta/2}\dd  \Pi_N  W_t \rangle + C_N \dd t
$$
where $C_N = [A^{\theta/2}\Pi_N W]_t = \sum_{0<|k|\le N} |k|^{2\theta}$. From this equation we easily obtain that for any  initial condition $v^N_0$ the process $(\|v^N_t\|_H)_{t\in[0,T]}$ is almost surely finite for any $T \ge 0$ which implies that the unique solution $(v^N_t)_{t \ge 0}$ can be extended to arbitrary intervals of time. Setting $u^N_t = v^N_t + (1-\Pi_N)Z_t$ we obtain a global solution of eq.~\eqref{eq:burgers-reg}. Moreover the diffusion $(u^N_t)_{t\ge 0}$ has  generator
$$
L_N \varphi(x) = L_0\varphi(x)+  \sum_{k\in\ZZ_0, |k|\le N} (F_N(x))_k D_k \varphi(x)
$$
where $L_0$ 
is the generator of the Ornstein--Uhlenbeck defined in eq.~\eqref{eq:generator-ou} and which satisfies the integration by parts formula
$
\mu [\varphi L_0 \varphi] = \mu[ \cE(\varphi)]
$ for $\varphi\in\cyl$.
This diffusion preserves the Gaussian measure $\mu$. Indeed if we take $u_0^N$ distributed according to the white noise $\mu$ we have that $((1-\Pi_N)Z_t)_{t \ge 0}$ is independent of $(v^N_t)_{t\ge 0}$. Moreover $Z_t$ has law $\mu$ of any $t\ge 0$ and an easy argument for the finite dimensional diffusion $(v^N_t)_{t\ge 0}$ shows that for any $t\ge0$ the random variable $v^N_t$ is distributed according to $\mu^N = (\Pi_N)_* \mu$: the push forward of the measure $\mu$ with respect to the projection $\Pi_N$.

We will use the fact that $u^N$ satisfy the  mild equation~\cite{DZ} 
\begin{equation}
\label{eq:burgers-reg-mild}
u^N_t = e^{-A^\theta t} u_0 +   \int_0^t e^{-A^\theta (t-s)} F_N(u^N_s) \dd s + A^{\theta/2} \int_0^t e^{-A^\theta (t-s)}  \dd W_s
\end{equation}
 where the stochastic convolution in the r.h.s is given by
$$
A^{\theta/2} \int_0^t e^{-A^\theta (t-s)}  \dd W_s = \sum_{k\in\ZZ_0} |k|^{\theta} e_k \int_0^t e^{-|k|^{2\theta} (t-s)} \dd \beta^k_s .
$$

\begin{lemma}
Let
$$
\cA_t^{N}=\int_0^t F_{N}(u^{N}_s) \dd s ,\qquad \tilde \cA_t^{N}=\int_0^t e^{-A^\theta(t-s)} F_{N}(u^{N}_s) \dd s . 
$$
and set $\sigma=(3/2-2\theta)_+$.
The family of laws of the processes $\{(u^N,\cA^N,\tilde \cA^N,W)\}_N$ is tight in the space of continuous functions with values in $\mathcal{X}=\cF L^{\infty,\sigma-\eps}\times \cF L^{\infty,3/2-\theta-\eps}\times \cF L^{\infty,3/2-2\theta-\eps}\times \cF L^{\infty,-\eps}$ for all small $\eps > 0$.
\end{lemma}

\begin{proof}
The estimate~\eqref{eq:basic-est-4-conv} in the previous section readily gives that for any small $\eps >0$ and sufficienly large $p$
$$
\expect_\mu\left[\sum_{k\in\ZZ_0} |k|^{-(3/2-2\theta+3\theta\eps) p} \left(|(\tilde \cA_t^{N}-\tilde \cA_s^{N})_k| \right)^p\right]
\lesssim_{p,\eps} \sum_{k\in\ZZ_0} |k|^{-\theta\eps p} |t-s|^{p \eps} \lesssim |t-s|^{p \eps}
$$
This estimates show that the family of processes $\{ \tilde \cA^{N} \}_{N}$ is tight in $C([0,T],\cF L^{\infty,\alpha})$ for $\alpha=3/2-2\theta+3\theta\eps$ and sufficiently small $\eps>0$.
An analogous argument using the estimate~\eqref{eq:basic-est-2}
 shows that the family of processes $\{  \cA^{N} \}_{N}$ is tight in $C^\gamma([0,T],\cF L^{\infty,\beta})$ for any $\gamma<1/2$ and $\beta < 3/2-\theta$.
It is not difficult to show that the stochastic convolution $\int_0^t e^{-A^\theta(t-s)} A^{\theta/2} \dd W_s$ belongs to $C([0,T],\cF L^{\infty,1-\theta-\eps})$ for all small  $\eps>0$. 
Taking into account the mild equation~\eqref{eq:burgers-reg-mild}
we find that the processes $\{(u^{N}_t)_{t\in[0,T]}\}_{N}$ are tight in $C([0,T],\cF L^{\infty,\sigma-\eps})$. 
\end{proof}

We are now ready to prove our main theorem on existence of (probabilistically weak) controlled solutions to the generalized stochastic Burgers equation.

\begin{theorem}
There exists a probability space and a quadruple of processes $(u,\cA,\tilde\cA,W)$ with continuous trajectories in $\mathcal{X}$ such that $W$ is a cylindrical Brownian motion in $H$, $u$ is a controlled process and they satisfy
\begin{equation}
\label{eq:limit-1}
u_t =  u_0 +  \cA_t  -  \int_0^t A^\theta u_s \dd s +  B W_t = e^{-A^\theta t} u_0 + \tilde \cA_t  + \int_0^t e^{-A^\theta(t-s)} B \dd W_s 
\end{equation}
where, as space distributions,
\begin{equation}
\label{eq:limit-4}
\cA_t = \lim_{M \to \infty}\int_0^t F_{M}(u_s) \dd s
\quad \text{ and } \quad 
\tilde \cA_t  = \int_0^t e^{-A^\theta(t-s)} \dd \cA_s .
\end{equation}
this last integral being defined as a Young integral.
\end{theorem}

\begin{proof}
Let us first prove~\eqref{eq:limit-4}. By tightness of the laws of $\{(u^N,\cA^N,\tilde\cA^N , W)\}_N$ in  $C(\mathbb{R};\mathcal{X})$ we can extract a subsequence which converges weakly (in the probabilistic sense) to a limit point in $C(\mathbb{R};\mathcal{X})$. By Skhorohod embedding theorem, up to a change of the probability space, we can assume that this subsequence which we call $\{N_n\}_{n\ge 1}$ converges almost surely to a limit $u = \lim_n u^{N_n} \in C(\mathbb{R};\mathcal{X})$. Then
$$
\int_0^t F_{M}(u_s) \dd s = \int_0^t (F_{M}(u_s) - F_{M}(u^{N_n}_s)) \dd s 
$$
$$\qquad
 + \int_0^t (F_{M}(u^{N_n}_s) - F_{N_n}(u^{N_n}_s)) \dd s
+ \int_0^t F_{N_n}(u^{N_n}_s) \dd s .
$$
But now, in  $C(\RR_+,\cF L^{\infty,3/2-\theta-\eps})$ we have the almost sure limit
$$
\lim_n \int_0^\cdot F_{N_n}(u^{N_n}_s) \dd s =\lim_n \cA^{N_n}_\cdot =  \cA_\cdot
$$
and, always almost surely in $C(\RR_+,\cF L^{\infty,3/2-\theta-\eps})$, we have also
$$
\lim_n \int_0^\cdot (F_{M}(u_s) - F_{M}(u^{N_n}_s)) \dd s = 0 ,
$$
since the functional $F_M$ depends only of a finite number of components of $u$ and $u^{N_n}$ and that we have the convergence of $u^{N_n}$ to $u$ in $C(\RR;\cF L^{\infty,\sigma-\eps})$ and thus distributionally uniformly in time.
Moreover, for all $k\in\ZZ_0$,
$$
\lim_M \sup_{N_n : M<N_n} \left\|\sup_{t\in[0,T]} \left| \int_0^t (F_{M}(u^{N_n}_s) - F_{N_n}(u^{N_n}_s))_k \dd s\right| \right\|_{L^p(\PP_\mu)}= 0.
$$
By the apriori estimates, $\cA^{N_n}$ converges to $\cA$ in $C^\gamma(\cF L^{\infty,3/2-\theta-\eps})$ for all $\gamma <1/2$ and $\eps > 0$ so that we can use Young integration to define $\int_0^t e^{-A^\theta(t-s)} \dd \cA^{N_n}_s$ as a space distribution and to obtain its distributional convergence (for example for each of its Fourier components) to $\int_0^t e^{-A^\theta(t-s)} \dd \cA^{N_n}_s$. At this point eq.~\eqref{eq:limit-1} is a simple consequence. The backward processes $\hat u^{N_n}_{t}=u^{N_n}_{T-t}$ and $\hat \cA^{N_n}_t = -\cA^{N_n}_{T-t}$ converge to $\hat u_{t}=u_{T-t}$ and $\hat \cA_t = -\cA_{T-t}$ respectively and moreover note that $\cA$ as a distributional process has trajectories which are H\"older continuous for any exponent smaller than $2\theta/(1+2\theta)>1/2$ as a consequence of the estimate~\eqref{eq:basic-est-4} and this directly implies that $\cA$ has zero quadratic variation. So $u$ is a controlled process in the sense of our definition.
\end{proof}

\section{Uniqueness for $\theta>5/4$}
\label{sec:uniq}

In this section we prove a simple pathwise uniqueness result for controlled solutions which is valid when $\theta > 5/4$. Note that to each controlled solution $u$  is naturally associated a cylindrical Brownian motion $W$ on $H$ given by the martingale part of the controlled decomposition~\eqref{eq:controlled-decomposition}. Pathwise uniqueness is then understood in the following sense.

\begin{definition}
SBE$_\theta$ has pathwise uniqueness if given two controlled processes $u,\tilde u\in\mathcal{R}_\theta$ on the same probability space which generate the same Brownian motion $W$ and such that $\tilde u_0 = u_0$ amost surely then there exists a negligible set $\mathcal{N}$ such that for all $\varphi\in\mathcal{S}$ and $t\ge 0$  $\{u_t(\varphi) \neq \tilde u_t(\varphi)\} \subseteq \mathcal{N}$.
\end{definition}

\begin{theorem}
\label{th:uniqueness}
The generalized stochastic Burgers equation has pathwise uniqueness when $\theta > 5/4$. 
\end{theorem}
\begin{proof}
Let $u$ be a controlled solution to the equation and let $u^N$ be the Galerkin
approximations defined above with respect to the cylindrical Brownian motion $W$ obtained from the martingale part of the decomposition of $u$ as a controlled process. We will prove that $u^N \to u$ almost surely in $C(\RR_+;\cF L^{2\theta-3/2-2\eps,\infty})$ for any small $\eps >0$. Since Galerkin
approximations have unique strong solutions we have $\tilde u^N = u^N$ almost surely and in the limit $\tilde u = u$  in $C(\RR_+;\cF L^{2\theta-3/2-2\eps,\infty})$ almost surely. This will imply the claim by taking as negligible set in the definition of pathwise uniqueness the set $\mathcal{N}=\{\sup_{t\ge 0}\|u_t-\tilde u_t\|_{\cF L^{2\theta-3/2-2\eps,\infty}}>0\}$. 
Let us  proceed to prove that $u^N \to u$.
By bilinearity,
$$
 F_N \left( u \right) - F_N \left( u^N \right)  = F_N ( \Pi_N u_s+u^N_s,\Delta^N_s)
$$
and the difference $\Delta^N = \Pi_N ( u - u^N )$
 satisfies the equation
$$
\Delta^N_t = \Pi_N \int_0^t e^{- A^{\theta} ( t - s )} 
   F_N ( u_s+u^N_s,\Delta^N_s) \dd s +\varphi^N_t
$$
where 
$$
\varphi^N_t = \int_0^t e^{- A^{\theta} \left( t - s \right)} \left( F_{} \left( u \right)
   - F_N \left( u^{} \right) \right) \dd s .
$$
Note that
\[ \| \sup_{t \in [ 0, T ]} | ( \varphi^N_t )_k | \|_{L^p (
   \mathbb{P}_{\mu} )} \lesssim_p \max(| k |^{1 - 2 \theta} N^{1
   / 2 - \theta},| k |^{3/2 - 2 \theta}) \]
which by interpolation gives
\[ \| \sup_{t \in [ 0, T ]} | ( \varphi^N_t )_k | \|_{L^p (
   \mathbb{P}_{\mu} )}\lesssim_p | k |^{3/2 - 2 \theta+\varepsilon} N^{-\varepsilon} \]
for any small $\eps >0$. Now let
$$
\Phi_N = \sup_{k\in\ZZ_0} \sup_{t \in [ 0, T ]} |k|^{2\theta-3/2-2\eps} | ( \varphi^N_t )_k | 
$$   
then
$$
\expect \sum_{N>1} N \Phi_N^p \le\sum_{N>1} N \sum_{k\in\ZZ_0}  \sup_{t \in [ 0, T ]} |k|^{p(2\theta-3/2-2\eps)}\expect  | ( \varphi^N_t )_k |^p 
$$
$$
\lesssim_p \sum_{N>1} N^{1-\eps p} \sum_{k\in\ZZ_0} |k|^{-p\eps}<+\infty
$$   
for $p$ large enough, which implies that almost surely
$
\Phi_N \lesssim_{p,\omega} N^{-1/p} 
$.
For the other term we have
$$
\sup_{t\in[0,T]}\left|\left( \int_0^t e^{- A^{\theta} \left( t - s \right)} F_N \left(
   \Pi_N u + u^N, \Delta_N \right) \dd s \right)_k\right| 
  \lesssim A_N |k|^{3/2-2\theta+2\eps} Q_T
$$   
where
   $
    A_N = \sup_{t \in \left[ 0, T \right]} \sup_k \left| k \right|^{2 \theta -
   3/2 - 2\varepsilon} \left| \left( \Delta^N_t \right)_k \right| $
and
$$
Q_T = \sup_{t\in[0,T]} |k|^{2\theta-1/2-2\eps} \int_0^t e^{- |k|^{2\theta} \left( t - s \right)} \sum_{q\in \ZZ_0} |(
   \Pi_N u_s + u^N_s)_q| |k-q|^{3/2-2\theta+2\eps}  \dd s 
 $$
This gives
\[ 
A_N \leqslant   Q_T A_N + \Phi_N.
 \] 
Since $3/2-2 \theta <-1 $ (that is $\theta > 5/4$), we have the estimate:
 $$
Q_T \lesssim  \sup_{t\in[0,T]} |k|^{2\theta-1/2-2\eps} \left[\int_0^t e^{- p' |k|^{2\theta} \left( t - s \right)}  \dd s  \right]^{1/p'} \left[ \int_0^T  \sum_{q\in \ZZ_0}\frac{ |(
   \Pi_N u_s + u^N_s)_q|^{p}}{ |k-q|^{-3/2+2\theta-2\eps}}  \dd s \right]^{1/p}
 $$
valid for some $p> 1$ (with $1/p'+1/p=1$). Then
$$ Q_T \lesssim   |k|^{2\theta-1/2-2\eps-2\theta/p'}  \left[ \int_0^T  \sum_{q\in \ZZ_0}\frac{ |(
   \Pi_N u_s + u^N_s)_q|^{p}}{ |k-q|^{-3/2+2\theta-2\eps}}  \dd s \right]^{1/p}
$$
and taking $p$ large enough such that $2\theta-1/2-2\eps-2\theta/p'\le 0$ we obtain
$$
Q_T \lesssim_p \left[ \int_0^T  \sum_{q\in \ZZ_0}\frac{ |(
   \Pi_N u_s + u^N_s)_q|^{p}}{ |k-q|^{-3/2+2\theta-2\eps}}  \dd s \right]^{1/p}
$$
By the stationarity of the processes $u$ and $u^N$ and the fact that their marginal laws are the white  noise we have
$$
\expect[ Q_T^p] \lesssim_p \int_0^T  \sum_{q\in \ZZ_0}\frac{\expect |(
   \Pi_N u_s + u^N_s)_q|^{p}}{ |k-q|^{-3/2+2\theta-2\eps}}  \dd s = T \sum_{q\in \ZZ_0}\frac{1}{ |k-q|^{-3/2+2\theta-2\eps}} \lesssim_p T
$$
Then by a simple Borel-Cantelli argument, almost surely $Q_{1/n} \lesssim_{p,\omega} n^{-1+1/p}$. Putting together the estimates for $\Phi_N$ and that for $Q_{1/n}$ we see that there exists a (random)  $T$ such that $C Q_T\le 1/2$ almost surely and that for this $T$:
$
A_N \leqslant  2 \Phi_N
$, 
which given the estimate on $\Phi_N$ implies that $A_N \to 0$ as $N\to\infty$ almost surely and that the solution of the equation is unique and is the
(almost-sure) limit of the Galerkin approximations.
\end{proof}

\section{Alternative equations}
\label{sec:alternative}
The technique of the present paper extends straighforwardly to some other modifications of the stochastic Burgers equation.

\subsection{Regularization of the convective term}
 Consider for example the equation
\begin{equation}
\label{eq:burgers-daprato}
\dd u_t = - A u_t \dd t + A^{-\sigma}F(A^{-\sigma} u_t) \dd t + B \dd W_t 
\end{equation}
which is the equation considered by Da~Prato, Debbussche and Tubaro in~\cite{DDT}. Letting $F_\sigma(x) = A^{-\sigma}F(A^{-\sigma} x)$, denoting by $H_\sigma$ the corresponding solution of the Poisson equation and following the same strategy as above we obtain the same bounds
$$
 \cE((H_{\sigma,N})^\pm_k)(x)\lesssim     \sum_{\substack{k_1,k_2 :k_1+k_2=k\\|k|,|k_1|,|k_2|\le N}}c_\sigma(k,k_1,k_2) |x_{k_2}|^2  
$$
where $c_\sigma(k,k_1,k_2) = |k|^{2-4\sigma}/[|k_1|^{4\sigma}|k_1|^{4\sigma}(|k_1|^{2}+|k_2|^{2})]$.
This quantity can then be bounded in terms of the sum
$$
I_{\sigma,N}(k) = \sum_{\substack{k_1,k_2 :k_1+k_2=k\\|k|,|k_1|,|k_2|\le N}}c_\sigma(k,k_1,k_2) \lesssim |k|^{1-12\sigma}
$$
From which we can reobtain similar bounds to those exploited above.
For example
$$
\left\|\int_0^t (e^{-A(t-s)}F_{\sigma,M}(u_s))_k \dd  s \right\|_{L^p(\PP_\mu)}\lesssim_p |k|^{-1/2-6\sigma} 
$$
And in particular we have existence of weak controlled solutions when $8\sigma+2>1$, that is $\sigma>-1/8$ and pathwise uniqueness when $-1/2-6\sigma<-1$ that is $\sigma> 1/12$. Which is an improvement over the result in~\cite{DDT} which has uniqueness for $\sigma>1/8$. 

\subsection{The Sasamoto--Spohn discrete model}
Another application of the above techniques is to the analysis of the discrete approximation to the stochastic Burgers equation proposed by Spohn and Sasamoto in~\cite{Spohn}. Their model is the following:
\begin{equation}
\label{eq:sasamoto-spohn}
\begin{split}
\dd u_j & = (2N+1) (u_j^2+u_j u_{j+1}-u_{j-1}u_j-u^2_{j-1})\dd t
\\
 & \qquad  + (2N+1)^2(u_{j+1}-2 u_j+u_{j-1})\dd t + (2N+1)^{3/2} (\dd B_j - \dd B_{j-1})
\end{split}
\end{equation}
for $j=1,\dots,2N+1$ with periodic boundary conditions $u_0=u_{2N+1}$ and where the processes $(B_j)_{j=1,\dots,2N+1}$ are a family of independents standard Brownian motions with $B_0=B_{2N+1}$. This model has to be tought as the discretization of the dynamic of the periodic velocity field $u(x)$ with $x\in(-\pi,\pi]$ sampled on a grid of mesh size $1/(2N+1)$,  that is $u_j = u(\xi^N_j)$ with $\xi^N_j = -\pi+2\pi(j/(2N+1))$. This fixes also the scaling factors for the different contributions to the dynamics if we want that, at least formally, this equation goes to a limit described by a SBE. Passing to Fourier variables $\hat u(k) = (2N+1)^{-1}\sum_{j=0}^{2N-1} e^{i \xi^N_j k} u_j$ for $k\in \ZZ^N$ with $\ZZ^N = \ZZ\cap [-N,N]$ and imposing that $\hat u(0)=0$, that is, considering the evolution only with zero mean velocity we get the system of ODEs:
$$
\dd \hat u_t(k) = F^\flat_N(\hat u_t)_k \dd t 
 -  |g_N(k)|^2  \hat u_t(k) \dd t + (2N+1)^{1/2} g_N(k) \dd \hat B_t(k)
$$
for $k\in \ZZ_0^N=\ZZ_0\cap [-N,N]$, where $g_N(k)=(2N+1)(1-e^{i k/(2N+1)})$,
$$
F^\flat_N(u_t)_k = \sum_{k_1,k_2\in\ZZ^N_0} \hat u_t(k_1)\hat u_t(k_2)[ g_N(k)-g_N(k)^*+g_N(k_1)-g_N(k_2)^*]
$$
and $(\hat B_\cdot(k))_{k\in\ZZ_0^N}$ is a family of centred complex Brownian motions such that $\hat B(k)^* = \hat B(-k)$ and with covariance $\expect \hat B_t(k) \hat B_t(-l) = \II_{k=l} t (2N+1)^{-1}$. 
If we then let $\beta(k) = (2N+1)^{1/2} \hat B(k)$ we obtain a family of complex BM of covariance 
$\expect  \beta_t(k)  \beta_t(-l) = t \II_{k=l} $. The generator $L^\flat_N$ of this stochastic dynamics is given by
$$
L^{\flat}_N \varphi( x) = \sum_{k\in\ZZ^N_0} F^\flat_N(x)_k  D_k \varphi( x)+L^{g_N,OU}_N \varphi (x)
$$
with $$ 
L^{g_N}_N \varphi( x) = \sum_{k\in\ZZ^N_0} |g_N(xk)|^2 (x_k D_{k}+ D_{-k}  D_k) \varphi( x)
$$
the generator of the OU process corresponding to the linear part associated with the multiplier $g_N$. It is easy to check that the complete dynamics preserves the (discrete) white noise measure, indeed
$$
\sum_{k\in\ZZ_0^N} x_{-k} F^\flat_N(x)_k = \sum_{\substack{k,k_1,k_2\in\ZZ_0^N\\k+k_1+k_2=0}}  x_{k}  x_{k_1} x_{k_2}[ g_N(k)^*-g_N(k)+g_N(k_1)-g_N(k_2)^*] =0
$$
since the symmetrization of the r.h.s. with respect to the permutations of the variables $k,k_1,k_2$ yields zero. Then defining suitable controlled process with respect to the linear part of this equation we can prove our apriori estimates on additive functionals which are now controlled by the quantity
$$
 \cE^{g_N}((H_{g_N,N})^\pm_k)(x)\lesssim     \sum_{\substack{k_1,k_2 :k_1+k_2=k\\|k|,|k_1|,|k_2|\le N}}c_{g_N}(k,k_1,k_2) |x_{k_2}|^2  
$$
with $c_{g_N}(k,k_1,k_2) = |g_N(k)|^2/[(|g_N(k_1)|^{2}+|g_N(k_2)|^{2})]$. Moreover noting that
$$
|g_N(k)|^2 = 2 (2N+1)^2(1-\cos(2\pi k/(2N+1)) \sim |k|^2
$$
uniformly  $N$, it is possible to estimate this energy in the same way we did before in the case $\theta=1$ and obtain  that the family of stationary solutions of equation~\eqref{eq:sasamoto-spohn} is tight in $C([0,T],\cF L^{\infty,-\eps})$ for all $\eps >0$. Moreover using the fact that $g_N(k) \to ik$ as $N\to \infty$ uniformly for bounded $k$ and that 
$$
\pi_M F^\flat_N(\pi_M x)_k = \sum_{k_1,k_2\in\ZZ^N_0} \II_{|k|,|k_1|,|k_2|\le M}  x_{k_1} x_{k_2}[ g_N(k)-g_N(k)^*+g_N(k_1)-g_N(k_2)^*]
$$
$$
\to 3 i k \sum_{k_1,k_2\in\ZZ_0} \II_{|k|,|k_1|,|k_2|\le M}  x_{k_1} x_{k_2} = 3 F_M(x)_k
$$
it is easy to check that any accumulation point is a controlled solution of the stochastic Burgers equations~\eqref{eq:burgers-theta}. 

\section{2d stochastic Navier-Stokes equation}
\label{sec:ns}

We consider the problem of stationary solutions to the 2d stochastic Navier-Stokes equation considered in~\cite{albeverio-cruzeiro} (see also~\cite{albeverio-ferrario}). We would like to deal with invariant measures obtained by formally taking the kinetic energy of the fluid and considering the associated Gibbs measure. However this measure is quite singular and we need a bit of hyperviscosity in the equation to make our estimates work.
 
\subsection{The setting}
Fix $\sigma>0$ and consider the following stochastic differential equation
\begin{equation}
\label{eq:ns}
\dd (u_{t})_k = - |k|^{2+2\sigma} (u_{t})_k \dd t +  B_k(u_t) \dd t + |k|^{\sigma} \dd \beta^k_t
\end{equation}
where $(\beta^k)_{k\in \Q}$ is a family of complex BMs for which $(\beta^k)^* = \beta^{-k}$ and $\expect[\beta^k \beta^q] = \II_{q+k=0}$, $u$ is a stochastic process with continuous trajectories in the space of distributions on the two dimensional torus $\TT^2$,
$$
 B_k(x) = \sum_{k_1+k_2=k} b(k,k_1,k_2) x_{k_1} x_{k_2}
$$
where $x: \Q \to \CC$ is such that $x_{-k} = x_k^*$ and
$$
b(k,k_1,k_2) = \frac{(k^\bot \cdot k_1)(k \cdot k_2)}{k^2}
$$
with $(\xi,\eta)^\bot = (\eta,-\xi) \in \RR^2$. Apart from the two-dimensional setting and the difference covariance structure of the linear part this problem has the same structure as the one dimensional stochastic Burgers equation we considered before. 
Note that to make sense of it (and in order to construct controlled solutions)
we can consider the  Galerkin approximations constructed as follows. Fix $N$ and solve the  problem finite dimensional problem
\begin{equation}
\label{eq:ns-N}
\dd (u^N_{t})_k = - |k|^{2+\sigma} (u^N_{t})_k \dd t +  B^N_k(u^N_t) \dd t + |k|^{-\sigma} \dd \beta^k_t
\end{equation}
for $k \in \ZZ^2_N = \{k \in \ZZ^2 : |k| \le N\}$, where
\begin{equation}
\label{eq:drift-N}
 B^N_k(x) =\II_{|k|\le N} \sum_{\substack{k_1+k_2=k\\|k_1|\le N, |k_2|\le N}} b(k,k_1,k_2) x_{k_1} x_{k_2} 
 \end{equation}
The generator of the process $u^N$ is given by
$
L^N\varphi(x) = L_0\varphi(x) + \sum_{k\in \Q} B^N_k(x) D_k \varphi(x)
$
where
$$
L^0 \varphi(x) = \frac12  \sum_{k\in \Q} |k|^{2\sigma}( D_{-k}D_k\varphi(x) -|k|^2 x_k D_k \varphi(x))  
$$
is the generator of a suitable OU flow. Note moreover that the kinetic energy of $u$ given by
$
E(x) = \sum_k |k|^{2} |x_k|^2 
$
is invariant under the flow generated by $B^N$. Moreover
$
D_{k} B^N_k(x) = 0 
$
since $x_k$ does not enter in the expression of $B^N_k(x)$, so the vectorfields $B^N$ leave also the measure
$
\prod_{k \in \QN}  dx_k 
$
invariant. Then  the  (complex) Gaussian measures
$$
\gamma(dx) = \prod_{k \in \Q} Z_k e^{-|k|^2 |x_k|^2} dx_k
$$
is invariant under  the flow generated by $B^N$. (This measure should be understood restricted to the set $\{x \in \CC^{\Q} : x_{-k} = \overline{x_k} \}$).
The measure $\gamma$ is also invariant for the $u^N$ diffusion since it is invariant for $B^N$ and for the OU process generated by $L^0$.
Intoduce standard Sobolev norms
$
\|x\|_\sigma^2 = \sum_{k \in \Q} |k|^{2\sigma} |x_k|^2
$
and denote with $H^\sigma$ the space of elements $x$ with $\|x \|_{\sigma}<\infty$.
The measure $\gamma$ is the Gaussian measure associated to $H^1$ and is supported on any $H^\sigma$ with $\sigma<0$
$$
\int \|x\|_\sigma^2 \gamma(dx) = \sum_{k \in \Q}  |k|^{2\sigma-2} < \infty
$$
so $(\gamma, H^1, \cap_{\eps < 0}H^\eps)$ is an abstract Wiener space in the sense of Gross. Note that the vectorfield  $B_k(x)$ in not defined on the support of $\gamma$. To give sense of controlled solutions to this equation we need to control
$$
 \cE((H_{N})^\pm_k)(x)\lesssim     \sum_{\substack{k_1,k_2 :k_1+k_2=k\\|k|,|k_1|,|k_2|\le N}}c_{\text{ns}}(k,k_1,k_2) |x_{k_2}|^2  
$$
with $c_{\text{ns}}(k,k_1,k_2) = |k_1|^{2\sigma} |k_1|^2|k_2|^2/(|k_1|^{2+2\sigma}+|k_2|^{2+2\sigma})^2$ and note that the stationary expectation of this term can be estimated by
$$
I_N(k) =  \sum_{\substack{k_1,k_2 :k_1+k_2=k\\|k|,|k_1|,|k_2|\le N}}c_{\text{ns}}(k,k_1,k_2) |k_2|^{-2} \lesssim \sum_{\substack{k_1,k_2 :k_1+k_2=k\\|k|,|k_1|,|k_2|\le N}}\frac{|k_1|^{2+2\sigma}}{ (|k_1|^{2+2\sigma}+|k_2|^{2+2\sigma})^2}\lesssim 
$$
$$
 \lesssim \sum_{\substack{k_1,k_2 :k_1+k_2=k\\|k|,|k_1|,|k_2|\le N}}\frac{1}{ |k_1|^{2+2\sigma}+|k_2|^{2+2\sigma}}\lesssim |k|^{-2\sigma}
$$
for any $\sigma > 0$. This estimate allows to apply our machinery and obtain stationary controlled solutions to this equation.


\end{document}